\documentclass[a4paper]{amsart}
\usepackage{amscd}
\usepackage{amsmath}
\usepackage{amssymb}
\usepackage{mathrsfs}
\usepackage{amsthm}
\usepackage{bbm} 
\usepackage{stmaryrd}

\usepackage[T1]{fontenc}

\newif\ifpdf
\ifx\pdfoutput\undefined
   \pdffalse        
\else
   \pdfoutput=1     
   \pdftrue
\fi

\ifpdf
   \usepackage[pdftex]{graphicx}
\else
   \usepackage{graphicx}
\fi

\numberwithin{equation}{section} \swapnumbers

\newtheorem{satz}{Satz}[section]

\newtheorem{theorem}[satz]{Theorem}
\newtheorem{proposition}[satz]{Proposition}
\newtheorem{corollary}[satz]{Corollary}
\newtheorem{lemma}[satz]{Lemma}

\newtheorem{definition}[satz]{Definition}

\newtheorem{remark}[satz]{Remark}

\newcommand{\bbr}{\mathbb{R}}
\newcommand{\bbe}{\mathbb{E}}
\newcommand{\bbn}{\mathbb{N}}
\newcommand{\bbp}{\mathbb{P}}

\newcommand{\bbx}{\mathbb{X}}

\newcommand{\calc}{\mathcal{C}}

\newcommand{\calf}{\mathscr{F}}

\newcommand{\calm}{\mathcal{M}}

\newcommand{\scrc}{\mathscr{C}}
\newcommand{\scrd}{\mathscr{D}}

\newcommand{\Lip}{{\rm Lip}}

\newcommand{\la}{{\langle}}
\newcommand{\ra}{{\rangle}}

\newcommand{\bfx}{\mathbf{X}}

\newcommand{\tr}{{\rm tr}}

\begin{document}

\title[Invariant submanifolds for rough differential equations]{Invariant submanifolds for solutions to rough differential equations}
\author{Stefan Tappe}
\date{2 October, 2024}
\address{Albert Ludwig University of Freiburg, Department of Mathematical Stochastics, Ernst-Zermelo-Stra\ss{}e 1, D-79104 Freiburg, Germany}
\email{stefan.tappe@math.uni-freiburg.de}
\thanks{The author gratefully acknowledges financial support from the Deutsche Forschungsgemeinschaft (DFG, German Research Foundation) -- project number 444121509.}
\begin{abstract}
In this paper we provide necessary and sufficient conditions for invariance of finite dimensional submanifolds for rough differential equations (RDEs) with values in a Banach space. Furthermore, we apply our findings to the particular situation of random RDEs driven by $Q$-Wiener processes and random RDEs driven by $Q$-fractional Brownian motion.
\end{abstract}

\maketitle\thispagestyle{empty}

\section{Introduction}

Consider a rough differential equation (RDE) of the form
\begin{align}\label{RDE}
\left\{
\begin{array}{rcl}
dY_t & = & f_0(Y_t) dt + f(Y_t) d \bfx_t
\\ Y_0 & = & \xi.
\end{array}
\right.
\end{align}
The state space of the RDE (\ref{RDE}) is a Banach space $W$, the driving signal $\bfx = (X,\bbx) \in \scrc^{\alpha}([0,T],V)$ is a rough path with values in another Banach space $V$ for some index $\alpha \in (\frac{1}{3},\frac{1}{2}]$, and the coefficients are suitable mappings $f_0 : W \to W$ and $f : W \to L(V,W)$.

Let $\calm$ be a finite dimensional $C^3$-submanifold of $W$. We say that the submanifold $\calm$ is (locally) invariant for the RDE (\ref{RDE}) if for each starting point $\xi \in \calm$ the solution $Y$ to (\ref{RDE}) stays (locally) on $\calm$. For stochastic equations the invariance of submanifolds has already been investigated in the literature; see, for example, \cite{Milian-manifold} for finite dimensional equations and \cite{Filipovic-inv, Nakayama, FTT-manifolds} for infinite dimensional equations driven by Wiener processes, and \cite{Ohashi} for infinite dimensional equations driven by fractional Brownian motion.

By now, there is a well-established theory of rough paths, including rough integration and RDEs; see, for example \cite{Friz-Hairer}. Some articles even deal with rough partial differential equations (RPDEs); see, for example \cite{Gubinelli-Tindel, Hairer, Teichmann, Hesse-Neamtu-local, Hesse-Neamtu-global, Tappe-rough}. Moreover, stochastic processes such as Brownian motion and fractional Brownian motion can be enhanced such that the theory of rough paths applies. This offers a pathwise approach to the study of path properties of stochastic processes and is tailor-made for stochastic invariance problems.

In this paper, we will clarify when a finite dimensional submanifold $\calm$ is locally invariant for the RDE (\ref{RDE}), and provide conditions in terms of the coefficients $f_0$ and $f$, which are necessary and sufficient; see Theorem \ref{thm-main} and Corollary \ref{cor-main}. Furthermore, we will apply these results to the particular situation of random RDEs driven by $Q$-Wiener processes (see Corollary \ref{cor-random-RDE-Ito}) and to random RDEs driven by $Q$-fractional Brownian motion (see Corollary \ref{cor-random-RDE-frac}), and compare our findings with those from the literature.

So far, there are only few references dealing with invariance results for RDEs; see, for example \cite{Coutin} for finite dimensional RDEs, the articles \cite{Ciotir, Nie, Xu, Trong} for equations driven by fractional Brownian motion, and \cite{Melnikov} for equations driven by a mixed process. Some papers (such as \cite{Cass, Armstrong}) develop a theory of rough integration and RDEs on manifolds. Let us also mention recent papers such as \cite{Kuehn, Ghani-1, Ghani-2}, which deal with invariant manifolds for RDEs from the perspective of dynamical systems.

The remainder of this paper is organized as follows. In Section \ref{sec-main-result} we present and prove our main result and its consequences. Afterwards, in Section \ref{sec-random-RDEs} we apply our findings to random RDEs; in particular to RDEs driven by infinite dimensional fractional Brownian motion. For convenience of the reader, we provide the required background about rough paths and rough differential equations in Appendix \ref{app-rough}, about finite dimensional submanifolds in Banach spaces in Appendix \ref{app-manifolds}, and about functions in Banach spaces in Appendix \ref{app-functions}.

\section{Presentation of the main result}\label{sec-main-result}

In this section we present and prove our main result as well as some consequences. Let $V$ and $W$ be Banach spaces. We fix a time horizon $T \in \bbr_+$. Let $\bfx = (X,\bbx) \in \scrc^{\alpha}([0,T],V)$ be a rough path for some index $\alpha \in (\frac{1}{3},\frac{1}{2}]$. Consider the RDE (\ref{RDE}) with a Lipschitz continuous mapping $f_0 : W \to W$ and a mapping $f : W \to L(V,W)$ of class $C_b^3$. According to Theorem \ref{thm-existence-RDE} for every $\xi \in W$ there exists a unique solution $(Y,Y') \in \scrd_X^{2 \alpha}([0,T],W)$ to the RDE (\ref{RDE}) with $Y_0 = \xi$. Let $\calm$ be a finite dimensional $C^3$-submanifold of $W$.

\begin{definition}
We call the submanifold $\calm$ \emph{locally invariant} for the RDE (\ref{RDE}) if for each $\xi \in \calm$ there exists $T_0 \in (0,T]$ such that $Y_t \in \calm$ for all $t \in [0,T_0]$, where $(Y,Y')$ denotes the solution to the RDE (\ref{RDE}) with $Y_0 = \xi$.
\end{definition}

\begin{definition}
We call the submanifold $\calm$ \emph{(globally) invariant} for the RDE (\ref{RDE}) if for each $\xi \in \calm$ we have $Y_t \in \calm$ for all $t \in [0,T]$, where $(Y,Y')$ denotes the solution to the RDE (\ref{RDE}) with $Y_0 = \xi$.
\end{definition}

\begin{theorem}\label{thm-main}
We assume that $X$ is truly rough and that $[\bfx]_t = x t$, $t \in [0,T]$ for some symmetric element $x \in V \otimes V$. Then the following statements are equivalent:
\begin{enumerate}
\item[(i)] The submanifold $\calm$ is locally invariant for the RDE (\ref{RDE}).

\item[(ii)] For each $y \in \calm$ we have
\begin{align}\label{tang-drift}
f_0(y) - \frac{1}{2} D f(y) f(y) x &\in T_y \calm,
\\ \label{tang-vol} f(y)v &\in T_y \calm, \quad v \in V.
\end{align}
\end{enumerate}
If the two equivalent statements (i) and (ii) are fulfilled and the submanifold $\calm$ is closed as a subset of $W$, then $\calm$ is even globally invariant for the RDE (\ref{RDE}).
\end{theorem}

\begin{remark}
In view of conditions (\ref{tang-drift}) and (\ref{tang-vol}), note that for each $y \in W$ we have
\begin{align*}
f(y) &\in L(V,W),
\\ Df(y) &\in L(W,L(V,W)),
\\ Df(y)f(y) &\in L(V,L(V,W)) \hookrightarrow L(V \otimes V,W).
\end{align*}
\end{remark}

\begin{remark}\label{rem-main}
Suppose that $V$ is a separable Hilbert space. Then we have $V \otimes V \simeq L_2(V)$, where $L_2(V)$ denotes the space of Hilbert Schmidt operators on $V$. Therefore, there are an orthonormal basis $\{ e_k \}_{k \in \bbn}$ of $V$ and a sequence $(\lambda_k)_{k \in \bbn} \subset \bbr$ with $\sum_{k=1}^{\infty} \lambda_k^2 < \infty$ such that
\begin{align*}
x = \sum_{k=1}^{\infty} \lambda_k ( e_k \otimes e_k ).
\end{align*}
Consequently, using Proposition \ref{prop-sum-in-drift} we can express the drift condition (\ref{tang-drift}) as
\begin{align*}
f_0(y) - \frac{1}{2} \sum_{k=1}^{\infty} \lambda_k D f_k(y) f_k(y) \in T_y \calm,
\end{align*}
where the mappings $f_k : W \to W$, $k \in \bbn$ are given by $f_k(y) := f(y) e_k$, $y \in W$.
\end{remark}

As an immediate consequence of Theorem \ref{thm-main} and Lemma \ref{lemma-geometric} we obtain the following result.

\begin{corollary}\label{cor-main}
Suppose that $\mathbf{X} \in \scrc_g^{\alpha}([0,T],V)$ is a weakly geometric rough path such that $X$ is truly rough. Then the following statements are equivalent:
\begin{enumerate}
\item[(i)] The submanifold $\calm$ is locally invariant for the RDE (\ref{RDE}).

\item[(ii)] For each $y \in \calm$ we have
\begin{align*}
f_0(y) &\in T_y \calm,
\\ f(y)v &\in T_y \calm, \quad v \in V.
\end{align*}
\end{enumerate}
If the two equivalent statements (i) and (ii) are fulfilled and the submanifold $\calm$ is closed as a subset of $W$, then $\calm$ is even globally invariant for the RDE (\ref{RDE}).
\end{corollary}

\begin{proof}[Proof of Theorem \ref{thm-main}]
(i) $\Rightarrow$ (ii): Let $\xi \in \calm$ be arbitrary. Furthermore, let $(Y,Y') \in \scrd_X^{2 \alpha}([0,T],W)$ be the solution to the RDE (\ref{RDE}) with $Y_0 = \xi$. Since $\calm$ is locally invariant for the RDE (\ref{RDE}), there exists $T_0 \in (0,T]$ such that $Y_t \in \calm$ for all $t \in [0,T_0]$. According to Proposition \ref{prop-global-para} there exist a mapping $\phi \in C_b^3(\bbr^m,W)$, an open neighborhood $O \subset \bbr^m$ of $\xi$ such that $\phi|_O : O \to U \cap \calm$ is a local parametrization around $\xi$, and a continuous linear operator $\ell \in L(W,\bbr^m)$ such that $\phi(\ell(y)) = y$ for all $y \in U \cap \calm$. By the continuity of $Y$ we may assume that $T_0$ is small enough such that $Y_t \in U \cap \calm$ for all $t \in [0,T_0]$. We define the new mappings
\begin{align*}
g_0 &:= \ell \circ f_0 \circ \phi : \bbr^m \to \bbr^m,
\\ g &:= \ell \circ f \circ \phi : \bbr^m \to L(V,\bbr^m),
\end{align*}
where, according to Lemma \ref{lemma-embedding-operators}, the linear operator $\ell$ is regarded as an element from $L(L(V,W),L(V,\bbr^m))$. By Proposition \ref{prop-Cb1-Lipschitz} the mapping $g_0$ is Lipschitz continuous. Furthermore, by Proposition \ref{prop-smooth-chain} and Proposition \ref{prop-smooth-linear} the mapping $g$ is of class $C_b^3$. Set $\eta := \ell(\xi) \in O$. We define the $O$-valued path $Z \in \calc^{\alpha}([0,T_0],\bbr^m)$ as $Z := \ell(Y)$. Noting that $Y = \phi(Z)$, we obtain
\begin{align*}
Z_t &= \ell(\xi) + \int_0^t \ell(f_0(Y_s)) ds + \int_0^t \ell(f(Y_s)) d\bfx_s
\\ &= \eta + \int_0^t \ell(f_0(\phi(Z_s))) ds + \int_0^t \ell(f(\phi(Z_s))) d\bfx_s, \quad t \in [0,T_0].
\end{align*}
Therefore, by Corollary \ref{cor-Gubinelli-der-mixed} we have $(Z,Z') \in \scrd_X^{2\alpha}([0,T_0],\bbr^m)$, where $Z' := g(Z)$, and it follows that $(Z,Z')$ is the local solution to the RDE (\ref{rough-ODE-Z-app}) with $Z_0 = \eta$ on the interval $[0,T_0]$. Moreover, by It\^{o}'s formula (Corollary \ref{cor-Ito}) we obtain
\begin{align*}
Y_t &= \xi + \int_0^t \bigg( D \phi(Z_s) g_0(Z_s) + \frac{1}{2} D^2 \phi(Z_s)(g(Z_s),g(Z_s)) x \bigg) ds
\\ &\quad + \int_0^t D \phi(Z_s)g(Z_s) d \bfx_s, \quad t \in [0,T_0].
\end{align*}
Since $Y$ is a solution to the RDE (\ref{RDE}) with $Y_0 = \xi$, by the Doob-Meyer result for rough integrals (Theorem \ref{thm-Doob-Meyer}) it follows that
\begin{align}\label{tang-proof-1}
f(Y) &= D \phi(Z) g(Z),
\\ \label{tang-proof-2} f_0(Y) &= D \phi(Z) g_0(Z) + \frac{1}{2} D^2 \phi(Z) (g(Z),g(Z))) x.
\end{align}
Since $\xi \in \calm$ was arbitrary, by (\ref{tang-proof-1}) we deduce (\ref{tang-vol}). By Proposition \ref{prop-decomposition} we have the decomposition
\begin{align*}
D f(Y) f(Y) = D \phi(Z) Dg(Z) g(Z) + D^2 \phi(Z) (g(Z),g(Z)).
\end{align*}
Therefore, by (\ref{tang-proof-2}) we obtain
\begin{align*}
f_0(Y) - \frac{1}{2} Df(Y) f(Y)x &= D \phi(z) g_0(Z) + \frac{1}{2} D^2 \phi(z) (g(Z),g(Z))) x
\\ &\quad - \frac{1}{2} Df(Y) f(Y)x
\\ &= D \phi(z) \bigg( g_0(Z) - \frac{1}{2} Dg(Z) g(Z) x \bigg).
\end{align*}
Since $\xi \in \calm$ was arbitrary, this proves (\ref{tang-drift}).

\noindent (ii) $\Rightarrow$ (i): Let $\xi \in \calm$ be arbitrary. Furthermore, let $(Y,Y') \in \scrd_X^{2 \alpha}([0,T],W)$ be the solution to the RDE (\ref{RDE}) with $Y_0 = \xi$. According to Proposition \ref{prop-global-para} there exist a mapping $\phi \in C_b^3(\bbr^m,W)$, an open neighborhood $O \subset \bbr^m$ of $\xi$ such that $\phi|_O : O \to U \cap \calm$ is a local parametrization around $\xi$, and a continuous linear operator $\ell \in L(W,\bbr^m)$ such that $\phi(\ell(y)) = y$ for all $y \in U \cap \calm$. Furthermore, for all $y \in U \cap \calm$ and all $w \in T_y \calm$ we have
\begin{align}\label{tangent-local}
w = D \phi(z) \ell(w),
\end{align}
where $z := \ell(y) \in O$. Thus, by (\ref{tang-drift}) and (\ref{tang-vol}), for all $y \in U \cap \calm$ we have
\begin{align}\label{f0-local}
f_0(y) - \frac{1}{2} D f(y) f(y) x &= D \phi(z) \ell \bigg( f_0(y) - \frac{1}{2} D f(y) f(y) x \bigg),
\\ \label{f-local} f(y)v &= D \phi(z) \ell (f(y)v), \quad v \in V,
\end{align}
where $z := \ell(y) \in O$. We define the new mapping
\begin{align*}
g := \ell \circ f \circ \phi : \bbr^m \to L(V,\bbr^m),
\end{align*}
where, according to Lemma \ref{lemma-embedding-operators}, the linear operator $\ell$ is regarded as an element from $L(L(V,W),L(V,\bbr^m))$. By Proposition \ref{prop-smooth-chain} and Proposition \ref{prop-smooth-linear} the mapping $g$ is of class $C_b^3$. Furthermore, by (\ref{f-local}) we have
\begin{align}\label{f-local-2}
f(y) = D\phi(z) g(z), \quad y \in U \cap \calm,
\end{align}
where $z := \ell(y) \in O$. We define $g_0 : \bbr^m \to \bbr^m$ as
\begin{align}\label{def-g0}
g_0(z) := \ell \bigg( f_0(y) - \frac{1}{2} D f(y) f(y) x + \frac{1}{2} D \phi(z) Dg(z) g(z) x \bigg),
\end{align}
where $y := \phi(z)$. Note that we can express $g_0$ as
\begin{align*}
g_0 = \ell \circ \bigg( f_0 \circ \phi - \frac{1}{2} {\rm ev}_x \circ \Big( B_1(Df \circ \phi, f \circ \phi) - B_3(D \phi,B_2(Dg,g)) \Big) \bigg),
\end{align*}
where the continuous bilinear operators
\begin{align*}
&B_1 : L(W,L(V,W)) \times L(V,W) \to L(V,L(V,W)) \hookrightarrow L(V \otimes V,W),
\\ &B_2 : L(\bbr^m,L(V,\bbr^m)) \times L(V,\bbr^m) \to L(V,L(V,\bbr^m)) \hookrightarrow L(V \otimes V, \bbr^m)
\\ &B_3 : L(\bbr^m,W) \times L(V \otimes V, \bbr^m) \to L(V \otimes V,W)
\end{align*}
are respectively given by the compositions $B_i(R,S) = R \circ S$, $i=1,2,3$, and ${\rm ev}_x : L(V \otimes V,W) \to W$ denotes the evaluation operator $A \mapsto Ax$, which is a continuous linear operator. Therefore, by Proposition \ref{prop-Cb1-Lipschitz} and Proposition \ref{prop-bilinear-Lipschitz} the mapping $g_0$ is Lipschitz continuous. Hence, by Theorem \ref{thm-existence-RDE} there exists a unique solution $(Z,Z') \in \scrd_X^{2\alpha}([0,T],\bbr^m)$ to the RDE (\ref{rough-ODE-Z-app}) with $Z_0 = \eta$, where $\eta := \phi^{-1}(\xi) \in O$. Since $O$ is open and $Z$ is continuous, there exists $T_0 \in (0,T]$ such that $Z_t \in O$ for all $t \in [0,T_0]$. We define the $U \cap \calm$-valued path $Y := \phi(Z)$. By It\^{o}'s formula (Corollary \ref{cor-Ito}) we obtain
\begin{align*}
Y_t &= \xi + \int_0^t \bigg( D \phi(Z_s) g_0(Z_s) + \frac{1}{2} D^2 \phi(Z_s)(g(Z_s),g(Z_s)) x \bigg) ds.
\\ &\quad + \int_0^t D \phi(Z_s)g(Z_s) d \bfx_s, \quad t \in [0,T_0].
\end{align*}
By Corollary \ref{cor-Gubinelli-der-mixed} we have $(Y,Y') \in \scrd_X^{2\alpha}([0,T_0],W)$ with $Y' = D \phi(Z) g(Z)$, and by (\ref{f-local-2}) it follows that $Y' = f(Y)$. Therefore, by Proposition \ref{prop-decomposition} we have the decomposition
\begin{align*}
D f(Y) f(Y) = D \phi(Z) Dg(Z) g(Z) + D^2 \phi(Z) (g(Z),g(Z)).
\end{align*}
Now, by (\ref{def-g0}), (\ref{f0-local}) and (\ref{tangent-local}) we deduce that
\begin{align*}
&D \phi(Z) g_0(Z) + \frac{1}{2} D^2 \phi(Z)(g(Z),g(Z)) x
\\ &= D \phi(Z) \ell \bigg( f_0(Y) - \frac{1}{2} D f(Y) f(Y) x + \frac{1}{2} D \phi(Z) Dg(Z) g(Z)x \bigg)
\\ &\quad + \frac{1}{2} D^2 \phi(Z)(g(Z),g(Z)) x
\\ &= f_0(Y) - \frac{1}{2} D f(Y) f(Y) x + \frac{1}{2} D \phi(Z) Dg(Z) g(Z) x + \frac{1}{2} D^2 \phi(Z)(g(Z),g(Z)) x
\\ &= f_0(Y).
\end{align*}
Consequently, the path $(Y,Y')$ is a local solution to the RDE (\ref{RDE}) with $Y_0 = \xi$.

\noindent \textit{Additional statement:} Let $\xi \in \calm$ be arbitrary, and let $(Y,Y') \in \scrd_X^{2 \alpha}([0,T],W)$ be the solution to the RDE (\ref{RDE}) with $Y_0 = \xi$. We define
\begin{align}\label{def-stopping-time}
T_0 := \inf \{ t \in [0,T] : Y_t \notin \calm \}.
\end{align}
Since the submanifold $\calm$ is locally invariant for the RDE (\ref{RDE}), we have $T_0 \in (0,T]$. Suppose that $T_0 < T$. Since $\calm$ is closed as a subset of $W$, by the continuity of $Y$ we have $Y_t \in \calm$ for all $t \in [0,T_0]$. Moreover, since the submanifold $\calm$ is locally invariant for the RDE (\ref{RDE}), there exists $S_0 \in (0,T]$ such that $Z_t \in \calm$ for all $t \in [0,S_0]$, where $(Z,Z') \in \scrd_X^{2 \alpha}([0,T],W)$ denotes the solution to the RDE (\ref{RDE}) with $Z_0 = Y_{T_0}$. By Lemma \ref{lemma-concat} we have $Y_t = Z_{t-T_0}$ for all $t \in [T_0,T]$, and hence $Y_t \in \calm$ for all $t \in [T_0,T_0+S_0]$ with $t \leq T$, which contradicts the definition (\ref{def-stopping-time}) of $T_0$.
\end{proof}

\section{Infinite dimensional random rough differential equations}\label{sec-random-RDEs}

In this section we will apply our invariance results to random RDEs; in particular to RDEs driven by infinite dimensional fractional Brownian motion. Let $(\Omega,\calf,(\calf_t)_{t \in \bbr_+},\bbp)$ be a filtered probability space satisfying the usual conditions. Furthermore, let $V$ be a separable Hilbert space, and let $W$ be a Banach space. Let $f_0 : W \to W$ and $f : W \to L(V,W)$ be mappings such that $f_0$ is Lipschitz continuous and $f$ is of class $C_b^3$.

In order to introduce the driving noise $\mathbf{X} = (X,\bbx)$, let us briefly review some results about infinite dimensional fractional Brownian motion, which has been studied, for example in \cite{Duncan, Grecksch}.

Let $Q \in L_1^{++}(V)$ be a nuclear, self-adjoint, positive definite linear operator. Recall that an $V$-valued centered Gaussian process $X$ is called a \emph{$Q$-fractional Brownian motion} with Hurst index $H \in (\frac{1}{3},\frac{1}{2}]$ if its covariance function is given by
\begin{align}\label{cov-fct-0}
\bbe[X_s \otimes X_t] = \frac{1}{2} \Big( s^{2H} + t^{2H} - |t-s|^{2H} \Big) Q, \quad s,t \in \bbr_+.
\end{align}
For what follows, let $X$ be a $Q$-fractional Brownian motion with Hurst index $H \in (\frac{1}{3},\frac{1}{2}]$. We also fix a time horizon $T \in \bbr_+$. From (\ref{cov-fct-0}) it follows that for all $v,w \in V$ we have
\begin{align}\label{cov-fct}
\bbe [ \la X_s,v \ra \la X_t,w \ra ] = \frac{1}{2} \la Qv,w \ra \Big( s^{2H} + t^{2H} - |t-s|^{2H} \Big), \quad s,t \in [0,T].
\end{align}
Furthermore, for each $\alpha \in (\frac{1}{3},\frac{1}{2})$ we have $\bbp$-almost surely $X \in \calc^{\alpha}([0,T],V)$. There exist an orthonormal basis $\{ e_k \}_{k \in \bbn}$ of $V$ and a sequence $(\lambda_k)_{k \in \bbn} \subset (0,\infty)$ with $\sum_{k=1}^{\infty} \lambda_k < \infty$ such that
\begin{align*}
Q e_k = \lambda_k e_k \quad \text{for all $k \in \bbn$.}
\end{align*}
Furthermore, the $Q$-fractional Brownian motion $X$ admits the series representation
\begin{align}\label{series-Wiener}
X_t = \sum_{k=1}^{\infty} \sqrt{\lambda_k} \beta_t^k e_k, \quad t \in [0,T]
\end{align}
with a sequence of independent real-valued fractional Brownian motions $(\beta^k)_{k \in \bbn}$ with Hurst index $H$; see, for example  \cite{Grecksch-Anh}, \cite{Duncan-Maslowski} or \cite{Grecksch}.
By the series representation (\ref{series-Wiener}) we have
\begin{align}\label{norm-Wiener}
|X_{s,t}|^2 = \sum_{k=1}^{\infty} \lambda_k |\beta_{s,t}^k|^2, \quad s,t \in [0,T].
\end{align}
We wish to show that a typical sample path of $X$ is truly rough. For this purpose, we prepare the following auxiliary result.

\begin{lemma}\label{lemma-inf-positive}
We have $\inf_{|v| = 1} \la Qv,v \ra > 0$.
\end{lemma}

\begin{proof}
We will show that the mapping $\Phi : V \to \bbr_+$, $\Phi(v) := \la Qv,v \ra$ is weakly continuous. Indeed, let $(v_n)_{n \in \bbn} \subset V$ and $v \in V$ be such that $\la v_n,w \ra \to \la v,w \ra$ for each $w \in V$. Since the sequence $(v_n)_{n \in \bbn}$ is bounded, there is a constant $M > 0$ such that $|v_n| \leq M$ for all $n \in \bbn$. Therefore, we have
\begin{align*}
| \la v_n,e_k \ra | \leq |v_n| \, |e_k| \leq M \quad \text{for all $n,k \in \bbn$.}
\end{align*}
Therefore, and since $\sum_{k=1}^{\infty} \lambda_k < \infty$, by Lebesgue's dominated convergence theorem we obtain for $n \to \infty$ the convergence
\begin{align*}
\Phi(v_n) = \la Q v_n, v_n \ra = \sum_{k=1}^{\infty} \lambda_k | \la v_n,e_k \ra |^2 \to \sum_{k=1}^{\infty} \lambda_k | \la v,e_k \ra |^2 = \la Qv,v \ra = \Phi(v),
\end{align*}
showing the weak continuity of $\Phi$. Since the unit sphere $\{ v \in V : |v| = 1 \}$ is weakly compact, the proof is complete.
\end{proof}

Now, we can generalize \cite[Example 2]{Friz-Shekhar} as follows. The proof is analogous.

\begin{proposition}\label{prop-truly-rough}
The $Q$-fractional Brownian motion $X$ is truly rough. More precisely, for each $\alpha \in (\frac{1}{3},H)$ there is a dense subset $D \subset [0,T)$ such that $\bbp$-almost surely
\begin{align*}
\limsup_{t \downarrow s} \frac{|\la v,X_{s,t} \ra|}{|t-s|^{2\alpha}} = \infty \quad \text{for all $s \in D$  and all $v \in V \setminus \{ 0 \}$.}
\end{align*}
\end{proposition}

\begin{proof}
We consider the mapping $\psi : (0,\frac{1}{3}) \to (0,\infty)$ defined as
\begin{align*}
\psi(h) := h^H \sqrt{ 2 \ln \ln \Big( \frac{1}{h} \Big) }.
\end{align*}
Let $s \in [0,T)$ be arbitrary. By the law of iterated logarithm for one-fractional Brownian motion (see \cite{Arcones} and \cite[Remark 2.3.3]{Viitasaari}), for each $k \in \bbn$ there exists a $\bbp$-nullset $N_k$ such that
\begin{align*}
\limsup_{t \downarrow s} \frac{|\beta_{s,t}^k|}{\psi(t-s)} = 1 \quad \text{on $N_k^c$.}
\end{align*}
Therefore, defining the $\bbp$-nullset $N := \bigcup_{k \in \bbn} N_k$, by (\ref{norm-Wiener}) we obtain
\begin{align*}
\bigg( \limsup_{t \downarrow s} \frac{|X_{s,t}|}{\psi(t-s)} \bigg)^2 \leq \sum_{k=1}^{\infty} \lambda_k \limsup_{t \downarrow s} \frac{|\beta_{s,t}^k|^2}{\psi(t-s)^2} = \tr(Q) \quad \text{on $N^c$.}
\end{align*}
Now, we set $c := \inf_{|v| = 1} \la Qv,v \ra$. By Lemma \ref{lemma-inf-positive} we have $c > 0$. Let $v \in V$ with $|v| = 1$ be arbitrary. By (\ref{cov-fct}) the process
\begin{align*}
Y := \frac{\la v,X \ra}{\sqrt{\la Qv,v \ra}}
\end{align*}
is a real-valued fractional Brownian motion with Hurst index $H$. Therefore, by the law of iterated logarithm for one-dimensional fractional Brownian motion we obtain $\bbp$-almost surely
\begin{align*}
\limsup_{t \downarrow s} \frac{|\la v, X_{s,t} \ra|}{\psi(t-s)} = \sqrt{\la Qv,v \ra} \, \limsup_{t \downarrow s} \frac{|Y_{s,t}|}{\psi(t-s)} = \sqrt{\la Qv,v \ra} \geq \sqrt{c}.
\end{align*}
Consequently, applying \cite[Thm. 2]{Friz-Shekhar} completes the proof.
\end{proof}

For the first application of our invariance result we choose the Hurst index $H = \frac{1}{2}$. Then $X$ is a $Q$-Wiener process; see, for example \cite{Atma-book, Da_Prato, Liu-Roeckner}. We consider the It\^{o}-enhanced $Q$-Wiener process $\mathbf{X} = (X,\bbx)$, where the L\'{e}vy area is given by
\begin{align*}
\bbx_{s,t} &= \sum_{j,k=1}^{\infty} \sqrt{\lambda_j \lambda_k} \bigg( \int_s^t \beta_{s,r}^j d \beta_r^k \bigg) (e_j \otimes e_k), \quad s,t \in [0,T].
\end{align*}
For the following result, we refer to \cite[Prop. 10.9]{Tappe-rough}.

\begin{proposition}\label{prop-Wiener-rough-path}
For each $\alpha \in (\frac{1}{3},\frac{1}{2})$ we have $\bbp$-almost surely $\mathbf{X} \in \scrc^{\alpha}([0,T],V)$.
\end{proposition}

Now, we choose an index $\alpha \in (\frac{1}{3},\frac{1}{2})$. By Proposition \ref{prop-truly-rough} and Proposition \ref{prop-Wiener-rough-path} there is a $\bbp$-nullset $N$ such that $\mathbf{X} = (X,\bbx) \in \scrc^{\alpha}([0,T],V)$ and $X$ is truly rough on $N^c$. As an immediate consequence of Theorem \ref{thm-main} and Remark \ref{rem-main} we obtain the following result.

\begin{corollary}\label{cor-random-RDE-Ito}
The following statements are equivalent:
\begin{enumerate}
\item[(i)] The submanifold $\calm$ is locally invariant for the random RDE (\ref{RDE}) on $N^c$.

\item[(ii)] For each $y \in \calm$ we have
\begin{align*}
f_0(y) - \frac{1}{2} \sum_{k=1}^{\infty} \lambda_k D f_k(y) f_k(y) &\in T_y \calm,
\\ f(y)v &\in T_y \calm, \quad v \in V.
\end{align*}
\end{enumerate}
If the two equivalent statements (i) and (ii) are fulfilled and the submanifold $\calm$ is closed as a subset of $W$, then $\calm$ is even globally invariant on $N^c$.
\end{corollary}

\begin{remark}
In the situation of Corollary \ref{cor-random-RDE-Ito} we obtain a flow on the submanifold $\calm$, given by the solutions $Y^{\xi}(\omega)$ for $\xi \in \calm$ and $\omega \in N^c$.
\end{remark}

\begin{remark}
If $W$ is also a separable Hilbert space, then there is a coincidence of the rough integral and the It\^{o} integral; see \cite[Prop. 10.12]{Tappe-rough}. Hence, in this situation Corollary \ref{cor-random-RDE-Ito} also applies to stochastic partial differential equation (SPDEs) and confirms the findings about invariant manifolds for It\^{o} SPDEs; see, for example \cite{Milian-manifold} for finite dimensional equations and \cite{Filipovic-inv, Nakayama} for infinite dimensional equations.
\end{remark}

For the second application of our invariance result we consider a general $Q$-fractional Brownian motion $X$ with some Hurst index $H \in (\frac{1}{3},\frac{1}{2}]$. For the following result see \cite[Prop. 11.2]{Tappe-rough}. An essential part of the statement is a consequence of \cite[Lemma 2.4]{Hesse-Neamtu-local}.

\begin{proposition}\label{prop-frac-rough-path}
There exists a L\'{e}vy area $\bbx$ such that $\bbp$-almost surely $\mathbf{X} := (X,\bbx) \in \scrc_g^{\alpha}([0,T],V)$ for each $\alpha \in (\frac{1}{3},H)$.
\end{proposition}

\begin{remark}
If $H = \frac{1}{2}$, then $X$ is a $Q$-Wiener process, and $\mathbf{X} = (X,\bbx)$ according to Proposition \ref{prop-frac-rough-path} is a Stratonovich-enhanced $Q$-Wiener process.
\end{remark}

By Proposition \ref{prop-truly-rough} and Proposition \ref{prop-frac-rough-path} there is a $\bbp$-nullset $N$ such that $\mathbf{X} = (X,\bbx) \in \scrc_g^{\alpha}([0,T],V)$ and $X$ is truly rough on $N^c$. As an immediate consequence of Corollary \ref{cor-main} we obtain the following result.

\begin{corollary}\label{cor-random-RDE-frac}
The following statements are equivalent:
\begin{enumerate}
\item[(i)] The submanifold $\calm$ is locally invariant for the random RDE (\ref{RDE}) on $N^c$.

\item[(ii)] For each $y \in \calm$ we have
\begin{align*}
f_0(y) &\in T_y \calm,
\\ f(y)v &\in T_y \calm, \quad v \in V.
\end{align*}
\end{enumerate}
If the two equivalent statements (i) and (ii) are fulfilled and the submanifold $\calm$ is closed as a subset of $W$, then $\calm$ is even globally invariant on $N^c$.
\end{corollary}

\begin{remark}
In the situation of Corollary \ref{cor-random-RDE-frac} we obtain a flow on the submanifold $\calm$, given by the solutions $Y^{\xi}(\omega)$ for $\xi \in \calm$ and $\omega \in N^c$.
\end{remark}

\begin{remark}
Note that Corollary \ref{cor-random-RDE-frac} is in line with invariance results for SPDEs driven by fractional Brownian motion; see, for example \cite{Ohashi}.
\end{remark}

\begin{appendix}

\section{Rough paths and rough differential equations}\label{app-rough}

In this appendix we provide the required results about rough paths and rough differential equations. For more details, we refer to \cite{Friz-Hairer}. Let $V, W$ be Banach spaces, and let $T \in \bbr_+$ be a time horizon. Furthermore, let $\bfx = (X,\bbx) \in \scrc^{\alpha}([0,T],V)$ be a rough path for some index $\alpha \in (\frac{1}{3},\frac{1}{2}]$. We recall that $\scrd_X^{2\alpha}([0,T],W)$ denotes the linear space of \emph{controlled rough paths} with values in $W$. For $(Y,Y') \in \scrd_X^{2\alpha}([0,T],W)$ the path $Y'$ is called a \emph{Gubinelli derivative} of $Y$ (with respect to $X$). The following straightforward result will be useful.

\begin{lemma}\label{lemma-reg-derivative-zero}
For each $Y \in \calc^{2 \alpha}([0,T],W)$ we have $(Y,0) \in \scrd_X^{2 \alpha}([0,T],W)$.
\end{lemma}

Now, let $(Y,Y') \in \scrd_X^{2 \alpha}([0,T],L(V,W))$ be a controlled rough path. Noting the canonical injection $L(V,L(V,W)) \hookrightarrow L(V \otimes V,W)$, we define the $W$-valued rough integral
\begin{align*}
\int_0^1 Y_s \, d \bfx_s := \lim_{|\Pi| \to 0} \sum_{[s,t] \in \Pi} ( Y_s X_{s,t} + Y_s' \bbx_{s,t} ).
\end{align*}

\begin{lemma}\cite[Thm. 4.10]{Friz-Hairer}\label{lemma-Gubinelli}
Let $(Y,Y') \in \scrd_X^{2 \alpha}([0,T],L(V,W))$ be a controlled rough path. Then we have
\begin{align*}
\bigg( \int_0^{\cdot} Y_s \, d \bfx_s, Y \bigg) \in \scrd_X^{2 \alpha}([0,T],W).
\end{align*}
\end{lemma}

As an immediate consequence of Lemmas \ref{lemma-reg-derivative-zero} and \ref{lemma-Gubinelli} we obtain the following result.

\begin{corollary}\label{cor-Gubinelli-der-mixed}
Let $Z \in C([0,T],W)$ be a continuous path, and let $(Y,Y') \in \scrd_X^{2 \alpha}([0,T],L(V,W))$ be a controlled rough path. Then we have
\begin{align*}
\bigg( \int_0^{\cdot} Z_s ds + \int_0^{\cdot} Y_s \, d \bfx_s, Y \bigg) \in \scrd_X^{2 \alpha}([0,T],W).
\end{align*}
\end{corollary}

In general, the Gubinelli derivative $Y'$ of a controlled rough path does not need to be unique. However, if $X$ is truly rough, then we have uniqueness of the Gubinelli derivative as a consequence of the following result. Recall that a path $X \in \calc^{\alpha}([0,T], V)$ is called \emph{truly rough} if there is a dense subset $D \subset [0,T)$ such that
\begin{align*}
\limsup_{t \downarrow s} \frac{|\la v,X_{s,t} \ra|}{|t-s|^{2\alpha}} = \infty \quad \text{for all $s \in D$  and all $v \in V \setminus \{ 0 \}$.}
\end{align*}

\begin{proposition}\cite[Prop. 6.4]{Friz-Hairer}\label{prop-Gubinelli-unique}
Suppose that $X$ is truly rough, and let
\begin{align*}
(Y,Y'), (Y,\tilde{Y}') \in \scrd_X^{2\alpha}([0,T],L(V,W))
\end{align*}
be two controlled rough paths. Then we have $Y' = \tilde{Y}'$.
\end{proposition}

Moreover, for truly rough paths the following Doob-Meyer type result holds true.

\begin{theorem}\cite[Thm. 6.5]{Friz-Hairer}\label{thm-Doob-Meyer}
Suppose that $X$ is truly rough. Let
\begin{align*}
Z,\tilde{Z} \in C([0,T],W)
\end{align*}
be two continuous paths, and let
\begin{align*}
(Y,Y'), (\tilde{Y},\tilde{Y}') \in \scrd_X^{2\alpha}([0,T],L(V,W))
\end{align*}
be two controlled rough paths. If
\begin{align*}
\int_0^t Z_s ds + \int_0^t Y_s d \mathbf{X}_s = \int_0^t \tilde{Z}_s ds + \int_0^t \tilde{Y}_s d \mathbf{X}_s, \quad t \in [0,T],
\end{align*}
then we have $Z = \tilde{Z}$ and $(Y,Y') = (\tilde{Y},\tilde{Y}')$.
\end{theorem}

Let $\bar{W}$ be another Banach space. For $(Y,Y') \in \scrd_X^{2 \alpha}([0,T],W)$ and a function $\varphi \in C_b^2(W,\bar{W})$ we denote by $f(Y) : [0,T] \to \bar{W}$ the path
\begin{align*}
f(Y)_t := f(Y_t), \quad t \in [0,T],
\end{align*}
and we denote by $f(Y)' : [0,T] \to L(V,\bar{W})$ the path
\begin{align*}
f(Y)_t' := Df(Y_t) Y_t', \quad t \in [0,T].
\end{align*}

\begin{lemma}\cite[Lemma 7.3]{Friz-Hairer}\label{lemma-trans-rough-path}
Let $(Y,Y') \in \scrd_X^{2 \alpha}([0,T],W)$ and $f \in C_b^2(W,\bar{W})$ be arbitrary. Then we have $(f(Y),f(Y)') \in \scrd_X^{2 \alpha}([0,T],\bar{W})$.
\end{lemma}

We define the \emph{bracket} $[ \mathbf{X} ] : [0,T]^2 \to {\rm Sym}(V \otimes V)$ as
\begin{align*}
[ \mathbf{X} ]_{s,t} := X_{s,t} \otimes X_{s,t} - 2 \, {\rm Sym}(\bbx_{s,t}), \quad s,t \in [0,T].
\end{align*}

Furthermore, we define the space $\scrc_g^{\alpha}([0,T],V)$ of \emph{weakly geometric rough paths} as the set of all $(X,\bbx) \in \scrc^{\alpha}([0,T],V)$ such that
\begin{align*}
{\rm Sym}(\bbx_{s,t}) = \frac{1}{2} X_{s,t} \otimes X_{s,t} \quad \text{for all $s,t \in [0,T]$.}
\end{align*}
Then the following auxiliary result is obvious.

\begin{lemma}\label{lemma-geometric}
If $\mathbf{X} \in \scrc_g^{\alpha}([0,T],V)$ is a weakly geometric rough path, then we have $[ \mathbf{X} ] = 0$.
\end{lemma}

\begin{theorem}[It\^{o}'s formula]\label{thm-Ito-rough}\cite[Thm. 7.7]{Friz-Hairer}
Let $F \in C_b^3(W,\bar{W})$ be arbitrary, and let $(Y,Y') \in \scrd_X^{2 \alpha}([0,T],W)$ a controlled rough path of the form
\begin{align*}
Y_t = Y_0 + \int_0^t Y_s' d \bfx_s + \Gamma_t, \quad t \in [0,T]
\end{align*}
for some controlled rough path $(Y',Y'') \in \scrd_X^{2 \alpha}([0,T],L(V,W))$ and some path $\Gamma \in \calc^{2 \alpha}([0,T],W)$. Then we have
\begin{align*}
F(Y_t) &= F(Y_0) + \int_0^t DF(Y_s) Y_s' d\bfx_{s} + \int_0^t DF(Y_s) d\Gamma_s
\\ &\quad + \frac{1}{2} \int_0^t D^2 F(Y_s)(Y_s',Y_s') d[\bfx]_s, \quad t \in [0,T].
\end{align*}
\end{theorem}

Now, we consider the $W$-valued rough differential equation (RDE)
\begin{align}\label{rough-ODE-app}
\left\{
\begin{array}{rcl}
dY_t & = & f_0(Y_t) dt + f(Y_t) d \bfx_t
\\ Y_0 & = & \xi
\end{array}
\right.
\end{align}
with mappings $f_0 : W \to W$ and $f : W \to L(V,W)$. We assume that $f_0$ is continuous, and that $f$ is of class $C_b^2$.

\begin{definition}
Let $\xi \in W$ be arbitrary. A path $(Y,Y') \in \scrd_X^{2 \alpha}([0,T_0],W)$ for some $T_0 \in (0,T]$ is called a \emph{local solution} to the RDE (\ref{rough-ODE-app}) with $Y_0 = \xi$ if $Y' = f(Y)$ and
\begin{align*}
Y_t = \xi + \int_0^t f_0(Y_s) ds + \int_0^t f(Y_s) d\mathbf{X}_s, \quad t \in [0,T_0].
\end{align*}
If we can choose $T_0 = T$, then we also call $(Y,Y')$ a \emph{(global) solution} to the RDE (\ref{rough-ODE-app}) with $Y_0 = \xi$.
\end{definition}

\begin{lemma}\label{lemma-solution-second-order}
Let $\xi \in W$ be arbitrary, and let $(Y,Y') \in \scrd_X^{2 \alpha}([0,T_0],W)$ be a local solution to the RDE (\ref{rough-ODE-app}) with $Y_0 = \xi$ for some $T_0 \in (0,T]$. Then we have $(Y',Y'') \in \scrd_X^{2 \alpha}([0,T_0],L(V,W))$ with $Y'' = Df(Y)f(Y)$.
\end{lemma}

\begin{proof}
Since $Y' = f(Y)$, this is an immediate consequence of Lemma \ref{lemma-trans-rough-path}.
\end{proof}

For the following existence and uniqueness result see, for example \cite[Cor. 9.19]{Tappe-rough}.

\begin{theorem}\label{thm-existence-RDE}
Suppose that $f_0$ is Lipschitz continuous, and that $f$ is of class $C_b^3$. Then for every $\xi \in W$ there exists a unique global solution $(Y,Y') \in \scrd_X^{2 \alpha}([0,T],W)$ to the RDE (\ref{rough-ODE-app}) with $Y_0 = \xi$.
\end{theorem}

The proof of the following auxiliary result concerning the concatenation of solutions is straightforward, and therefore omitted.

\begin{lemma}\label{lemma-concat}
Suppose that $f_0$ is Lipschitz continuous, and that $f$ is of class $C_b^3$. Let $\xi \in W$ be arbitrary, and denote by $(Y,Y') \in \scrd_X^{2 \alpha}([0,T],W)$ the solution to the RDE (\ref{rough-ODE-app}) with $Y_0 = \xi$. Furthermore, let $T_0 \in (0,T)$ be arbitrary, and denote by $(Z,Z') \in \scrd_X^{2 \alpha}([0,T],W)$ the solution to the RDE (\ref{rough-ODE-app}) with $Z_0 = Y_{T_0}$. Then we have $Y_t = Z_{t-T_0}$ for all $t \in [T_0,T]$.
\end{lemma}

Now, let $E$ be another Banach space. We consider the $E$-valued RDE
\begin{align}\label{rough-ODE-Z-app}
\left\{
\begin{array}{rcl}
dZ_t & = & g_0(Z_t) dt + g(Z_t) d \bfx_t
\\ Z_0 & = & \eta
\end{array}
\right.
\end{align}
with mappings $g_0 : E \to E$ and $g : E \to L(V,E)$. We assume that $g_0$ is continuous, and that $g$ is of class $C_b^2$.

\begin{corollary}\label{cor-Ito}
We assume that $[\bfx]_t = x t$, $t \in [0,T]$ for some symmetric element $x \in V \otimes V$. Let $\eta \in E$ be arbitrary, and let $(Z,Z') \in \scrd_X^{2\alpha}([0,T],E)$ be a solution to the RDE (\ref{rough-ODE-Z-app}) with $Z_0 = \eta$. Furthermore, let $\phi \in C_b^3(E,W)$ be arbitrary. Then we have
\begin{align*}
\phi(Z_t) &= \phi(\eta) + \int_0^t \bigg( D \phi(Z_s) g_0(Z_s) + \frac{1}{2} D^2 \phi(Z_s)(g(Z_s),g(Z_s)) x \bigg) ds
\\ &\quad + \int_0^t D \phi(Z_s)g(Z_s) d \bfx_s, \quad t \in [0,T].
\end{align*}
\end{corollary}

\begin{proof}
This follows from Theorem \ref{thm-Ito-rough} and Lemma \ref{lemma-solution-second-order}.
\end{proof}

The following decomposition (\ref{decomp}) may be compared with \cite[Prop. 6.1.3]{fillnm}, where a similar decomposition is presented in a differential geometric context for vector fields.

\begin{proposition}\label{prop-decomposition}
Suppose that $X$ is truly rough. Let $(Z,Z') \in \scrd_X^{2\alpha}([0,T],E)$ be a controlled rough path such that $Z' = g(Z)$ for some $g \in C_b^2(E,L(V,E))$, and let $(Y,Y') \in \scrd_X^{2\alpha}([0,T],W)$ be a controlled rough path such that $Y' = f(Y)$ for some $f \in C_b^2(W,L(V,W))$. Furthermore, we assume that $Y = \phi(Z)$ for some mapping $\phi \in C_b^3(E,W)$. Then we have the decomposition
\begin{align}\label{decomp}
D f(Y) f(Y) = D \phi(Z) Dg(Z) g(Z) + D^2 \phi(Z) (g(Z),g(Z)).
\end{align}
\end{proposition}

\begin{proof}
By Lemma \ref{lemma-trans-rough-path} we have $(Y',Y'') \in \scrd_X^{2 \alpha}([0,T],L(V,W))$ with
\begin{align*}
Y'' = Df(Y)Y' = Df(Y)f(Y).
\end{align*}
Consider the mapping $\varphi : E \to L(V,W)$ given by
\begin{align*}
\varphi(z) := D\phi(z) g(z), \quad z \in E.
\end{align*}
Note that $\varphi = B(D \phi,g)$, where $B : L(E,W) \times L(V,E) \to L(V,W)$ denotes the continuous bilinear operator given by $B(R,S) = R \circ S$. By the product rule (Proposition \ref{prop-Leibniz}) we obtain that $\varphi$ is of class $C_b^2$, and that for all $z,v \in E$ we have
\begin{align*}
D \varphi(z)v &= B(D^2 \phi(z)v,g(z)) + B(D \phi(z), Dg(z)v)
\\ &= D^2 \phi(z)(v,g(z)) + D \phi(z) D g(z) v.
\end{align*}
By Lemma \ref{lemma-trans-rough-path} we have $(\phi(Z),\phi(Z)') \in \scrd_X^{2\alpha}([0,T],W)$, where
\begin{align*}
\phi(Z)' = D \phi(Z) Z' = D \phi(Z) g(Z).
\end{align*}
Hence, by the uniqueness of the Gubinelli derivative (Proposition \ref{prop-Gubinelli-unique}) we obtain
\begin{align*}
Y' = \phi(Z)' = D \phi(Z) g(Z) = \varphi(Z).
\end{align*}
Furthermore, by Lemma \ref{lemma-trans-rough-path} we have $(\varphi(Z),\varphi(Z)') \in \scrd_X^{2\alpha}([0,T],L(V,W))$, where
\begin{align*}
\varphi(Z)' = D \varphi(Z) Z'.
\end{align*}
Thus, by the uniqueness of the Gubinelli derivative (Proposition \ref{prop-Gubinelli-unique}) we obtain
\begin{align*}
Df(Y)f(Y) &= Y'' = \varphi(Z)' = D \varphi(Z) Z' = D \phi(Z) D g(Z) Z' + D^2 \phi(Z)(Z',g(Z))
\\ &= D \phi(Z) D g(Z) g(Z) + D^2 \phi(Z)(g(Z),g(Z)),
\end{align*}
completing the proof.
\end{proof}

\section{Finite dimensional submanifolds in Banach spaces}\label{sec-manifolds-general}\label{app-manifolds}

In this appendix we provide the required results about finite dimensional submanifolds in Banach spaces; see \cite[Sec. 6]{fillnm} for further details. Let $W$ be a Banach space. Furthermore, let $m,k \in \bbn$ be positive integers.

\begin{definition}
A subset $\calm \subset W$ is called an $m$-dimensional \emph{$C^k$-submanifold} of $W$ if for every $y \in \calm$ there exist an open neighborhood $U \subset H$ of $y$, an open set $O \subset \mathbb{R}^m$ and a mapping $\phi \in C^k(O,W)$ such that:
\begin{enumerate}
\item The mapping $\phi : O \rightarrow U \cap \calm$ is a homeomorphism.

\item The linear operator $D \phi(z) \in L(\bbr^m,W)$ is one-to-one for all $z \in O$.
\end{enumerate}
The mapping $\phi$ is called a \emph{local parametrization} of $\calm$ around $y$.
\end{definition}

For what follows, let $\calm$ be an $m$-dimensional $C^k$-submanifold of $W$.

\begin{lemma}\cite[Lemma~6.1.1]{fillnm}\label{lemma-change-para}
Let $\phi_i : V_i \rightarrow U_i \cap \calm$, $i=1,2$ be two local parametrizations of $\calm$ with $X := U_1 \cap U_2 \cap \calm \neq \emptyset$. Then the mapping
\begin{align*}
\varphi := \phi_1^{-1} \circ \phi_2 : \phi_2^{-1}(X) \rightarrow \phi_1^{-1}(X)
\end{align*}
is a $C^k$-diffeomorphism.
\end{lemma}

\begin{definition}\label{def-tang-raum}
Let $y \in \calm$ be arbitrary. The \emph{tangent space} of $\calm$ to $y$ is the subspace
\begin{align*}
T_y \calm := D \phi(z)\mathbb{R}^m,
\end{align*}
where $z := \phi^{-1}(y)$, and $\phi : V \rightarrow U \cap \calm$ denotes a local parametrization of $\calm$ around $y$.
\end{definition}

\begin{remark}
By Lemma \ref{lemma-change-para} the Definition \ref{def-tang-raum} of the tangent space does not depend on the choice of the parametrization.
\end{remark}

\begin{proposition}\label{prop-global-para}
For each $y_0 \in \calm$ there exist a mapping $\phi \in C_b^k(\bbr^m,W)$, an open neighborhood $O \subset \bbr^m$ of $y_0$ such that $\phi|_O : O \to U \cap \calm$ is a local parametrization around $y_0$, and a continuous linear operator $\ell \in L(W,\bbr^m)$ such that $\phi(\ell(y)) = y$ for all $y \in U \cap \calm$. Furthermore, for all $y \in U \cap \calm$ and all $w \in T_y \calm$ we have $w = D\phi(z) \ell(w)$, where $z := \ell(y) \in O$.
\end{proposition}

\begin{proof}
This follows from \cite[Remark 6.1.1, Prop. 6.1.2 and Lemma 6.1.3]{fillnm}.
\end{proof}

\section{Functions in Banach spaces}\label{app-functions}

In this appendix we provide the required results about functions in Banach spaces; in particular smooth functions, Lipschitz continuous functions and linear operators. Further details can be found, for example in \cite{Abraham}. For what follows, let $E,F,G,H$ be Banach spaces, and let $k \in \bbn_0$ be a nonnegative integer.

\begin{proposition}\cite[Prop. 2.4.8]{Abraham}\label{prop-Cb1-Lipschitz}
Every mapping $f \in C_b^1(E,F)$ is Lipschitz continuous.
\end{proposition}

\begin{proposition}[Chain rule]\label{prop-smooth-chain}
Let $f : E \to F$, $g : F \to G$ be differentiable mappings. Then the following statements are true:
\begin{enumerate}
\item[(a)] The composition $g \circ f : E \to G$ is also differentiable, and we have
\begin{align*}
D (g \circ f)(x) = Dg(f(x)) \circ Df(x), \quad x \in E.
\end{align*}
\item[(b)] If $f$ and $g$ are of class $C^k$, then $g \circ f$ is also of class $C^k$.

\item[(c)] If $f$ and $g$ are of class $C_b^k$, then $g \circ f$ is also of class $C_b^k$.
\end{enumerate}
\end{proposition}

\begin{proof}
Parts (a) and (b) follow from \cite[Thm. 2.4.3]{Abraham}, and part (c) is a consequence of the higher order chain rule; see \cite[p. 88]{Abraham}.
\end{proof}

\begin{proposition}\label{prop-smooth-linear}
Let $f : E \to F$ be differentiable, and let $A \in L(F,G)$ be a continuous linear operator. Then the following statements are true:
\begin{enumerate}
\item[(a)] The composition $A \circ f : E \to G$ is also differentiable.

\item[(b)] If $f$ is of class $C^k$, then $A \circ f$ is also of class $C^k$, and for all $j=0,\ldots,k$ we have
\begin{align}\label{comp-linear-f}
D^j(A \circ f)(x) = A \circ D^j f(x), \quad x \in E.
\end{align}
\item[(c)] If $f$ is of class $C_b^k$, then $A \circ f$ is also of class $C_b^k$.
\end{enumerate}
\end{proposition}

\begin{proof}
Parts (a) and (b) are immediate consequences of the chain rule (see Proposition \ref{prop-smooth-chain}), and part (c) follows from the formula (\ref{comp-linear-f}).
\end{proof}

\begin{proposition}[Leibniz or product rule]\label{prop-Leibniz}
Let $f : E \to F$, $g : E \to G$ be differentiable mappings, and let $B \in L^{(2)}(F \times G,H)$ be a continuous bilinear operator. We define the new mapping
\begin{align*}
B(f,g) : E \to H, \quad x \mapsto B(f(x),g(x)).
\end{align*}
Then the following statements are true:
\begin{enumerate}
\item[(a)] The mapping $B(f,g)$ is also differentiable, and we have
\begin{align}\label{Leibniz-rule}
D(B(f,g))(x)v = B(Df(x)v,g(x)) + B(f(x),Df(x)v), \quad x,v \in E.
\end{align}
\item[(b)] If $f$ and $g$ are of class $C^k$, then $B(f,g)$ is also of class $C^k$.

\item[(c)] If $f$ and $g$ are of class $C_b^k$, then $B(f,g)$ is also of class $C_b^k$.
\end{enumerate}
\end{proposition}

\begin{proof}
Parts (a) and (b) are immediate consequences of \cite[Thm. 2.4.4]{Abraham}, and part (c) is a consequence of the multidimensional Leibniz rule, which follows inductively from (\ref{Leibniz-rule}).
\end{proof}

\begin{proposition}\label{prop-bilinear-Lipschitz}
Let $f : E \to F$, $g : E \to G$ be Lipschitz continuous and bounded mappings, and let $B \in L^{(2)}(F \times G,H)$ be a continuous bilinear operator. We define the new mapping
\begin{align*}
B(f,g) : E \to H, \quad x \mapsto B(f(x),g(x)).
\end{align*}
Then $B(f,g)$ is also Lipschitz continuous and bounded.
\end{proposition}

\begin{proof}
For all $x \in E$ we have
\begin{align*}
|B(f(x),g(x))| \leq |B| \, |f(x)| \, |g(x)| \leq |B| \, \| f \|_{\infty} \| g \|_{\infty},
\end{align*}
and for all $x,y \in E$ we have
\begin{align*}
&|B(f(x),g(x)) - B(f(y),g(y))|
\\ &= |B(f(x),g(x)) - B(f(x),g(y)) + B(f(x),g(y)) - B(f(y),g(y))|
\\ &\leq |B(f(x),g(x)-g(y))| + |B(f(x)-f(y),g(y))|
\\ &\leq |B| \, | f(x) | \, |g(x)-g(y)| + |B| \, |f(x)-f(y)| \, |g(y)|
\\ &\leq |B| \big( \| f \|_{\infty} \| g \|_{\Lip} + \| f \|_{\Lip} \| g \|_{\infty} \big) |x-y|,
\end{align*}
completing the proof.
\end{proof}

\begin{lemma}\label{lemma-embedding-operators}
We have the canonical injection $L(F,G) \hookrightarrow L(L(E,F),L(E,G))$.
\end{lemma}

\begin{proof}
Consider the linear operator
\begin{align*}
\Phi : L(F,G) \to L(L(E,F),L(E,G)), \quad B \mapsto (A \mapsto BA).
\end{align*}
Then we have
\begin{align*}
| \Phi |_{L(L(E,F),L(E,G))} &= \sup_{| A | \leq 1} | BA |_{L(E,G)}
\\ &\leq \sup_{| A | \leq 1} | B |_{L(F,G)} | A |_{L(E,F)} \leq | B |_{L(F,G)},
\end{align*}
showing that $\Phi$ is continuous. Let $B \in L(F,G)$ with $\Phi(B) = 0$ be arbitrary. Then we have $BA = 0$ for all $A \in L(E,F)$, and hence $By = 0$ for all $y \in F$, showing that $\Phi$ is one-to-one.
\end{proof}

\begin{lemma}\label{lemma-diff-fixed-x}
Let $f : F \to L(E,F)$ be a mapping of class $C^1$. For any $x \in E$ we define the new mapping
\begin{align*}
f_x : F \to F \quad y \mapsto f(y)x.
\end{align*}
Then the mapping $f_x$ is also of class $C^1$, and we have
\begin{align*}
Df_x(y)v = (Df(y)v)x, \quad y,v \in F.
\end{align*}
\end{lemma}

\begin{proof}
Let ${\rm ev}_x : L(E,F) \to F$ be the evaluation operator $A \mapsto Ax$. Then ${\rm ev}_x$ is a continuous linear operator, and we have $f_x = {\rm ev}_x \circ f$. Hence, by the chain rule (see Proposition \ref{prop-smooth-linear}), the mapping $f$ is of class $C^1$, and for all $y,v \in F$ we obtain
\begin{align*}
Df_x(y)v = D({\rm ev}_x \circ f)(y)v = {\rm ev}_x \circ Df(y)v = (Df(y)v)x,
\end{align*}
completing the proof.
\end{proof}

For the next result recall also that
\begin{align}\label{identification}
L(E,L(E,F)) \cong L^2(E \times E,F).
\end{align}

\begin{proposition}\label{prop-sum-in-drift}
Let $f : F \to L(E,F)$ be of class $C^1$. Then we have
\begin{align*}
Df(y)f(z)(v,w) = D f_w(y) f_v(z)
\end{align*}
for all $y,z \in F$ and all $v,w \in E$.
\end{proposition}

\begin{proof}
Using the identification (\ref{identification}) and Lemma \ref{lemma-diff-fixed-x} we obtain
\begin{align*}
Df(y)f(z)(v,w) = (D f(y)f(z) v) w = (D f(y) f_v(z))w = D f_w(y) f_v(z),
\end{align*}
completing the proof.
\end{proof}

\end{appendix}

\end{document}